\documentclass[11pt]{amsart}

\usepackage{amssymb}
\usepackage{setspace}
\usepackage{tikz-cd}
\usepackage{pict2e}
\usepackage{amsmath}
\usepackage{mathrsfs}
\usepackage{amsthm}
\usepackage{verbatim}
\usepackage{graphicx}
\usepackage{epstopdf}
\usepackage{epsfig}
\usepackage{color}
\usepackage{bm}
\usepackage{chngpage}
\usepackage{graphicx}
\usepackage{multirow}
\usepackage{multicol}
\usepackage{eucal}
\usepackage{hyperref}
\usepackage{graphicx}

\makeatletter
\newcommand*\bigcdot{\mathpalette\bigcdot@{.5}}
\newcommand*\bigcdot@[2]{\mathbin{\vcenter{\hbox{\scalebox{#2}{$\m@th#1\bullet$}}}}}
\makeatother

\definecolor{red}{rgb}{0.8,0,0}
\definecolor{darkorange}{rgb}{1,0.4,0}
\definecolor{lightorange}{rgb}{1,0.6, 0}
\definecolor{yellow}{rgb}{1,0.8, 0}

\newtheorem{theorem}{Theorem}
\newtheorem{remark}{Remark}

\newtheorem{lemma}{Lemma}

\newcommand\NED{\operatorname{NED}}
\newcommand\BDM{\operatorname{BDM}}
\newcommand\WF{\operatorname{WF}}
\newcommand\CT{\operatorname{CT}}
\newcommand\RT{\operatorname{RT}}

\newcommand\grad{\operatorname{grad}}
\renewcommand\div{\operatorname{div}}

\newcommand\curl{\operatorname{curl}}
\newcommand\rot{\operatorname{rot}}
\newcommand\ran{\mathcal{R}}
\renewcommand\ker{\mathcal{N}}

\newcommand\dd{\,\mathrm{d}}

\makeatletter
\@addtoreset{equation}{section}
\makeatother

\usepackage{geometry}
\begin{document}

\title{Partially discontinuous nodal finite elements for $H(\curl)$ and $H(\div)$}

\author{Jun Hu \and  Kaibo Hu \and Qian Zhang}
\address{ LMAM and School of Mathematical Sciences, Peking University, Beijing 100871, P. R. China}
\email{hujun@math.pku.edu.cn}
 \address{Mathematical Institute,
University of Oxford, Andrew Wiles Building, Radcliffe Observatory Quarter, 
Oxford, OX2 6GG, UK}
\email{Kaibo.Hu@maths.ox.ac.uk}
\address{Department of Mathematical Sciences, Michigan Technological University, Houghton, MI 49931, USA.}
\email{qzhang15@mtu.edu}

\date{}

\maketitle

\begin{abstract}
We investigate discretization of $H(\curl)$ and $H(\div)$ in two and three space dimensions by partially discontinuous nodal finite elements, i.e., vector-valued Lagrange finite elements with discontinuity in certain directions. These spaces can be implemented as a combination of continuous and discontinuous Lagrange elements and fit in de~Rham complexes. We construct well-conditioned nodal bases. 
\end{abstract}

\tableofcontents
\section{Introduction}

The edge elements  \cite{Nedelec.J.1980a,Nedelec.J.1986a} have achieved tremendous success in  computational electromagnetism. These elements fix the notorious spurious solutions in both eigenvalue and source problems  \cite{arnold2018finite,bossavit1990solving,sun1995spurious}. The key of this success is that the N\'{e}d\'{e}lec element fits in a finite element de~Rham complex (c.f., \cite{sun1995spurious}).  
Even though discrete differential forms \cite{Arnold.D;Falk.R;Winther.R.2006a,bossavit1988whitney,hiptmair1999canonical} have been a mature solution for many problems in computational electromagnetism, there are still interests in schemes with proper modifications of the vector-valued Lagrange elements for both theoretical and practical reasons. A theoretical question is what elements can solve the Maxwell equations without spurious solutions. In \cite{sun1995spurious} the authors showed that the vector-valued Lagrange element on the Powell-Sabin mesh avoids spurious eigenmodes since the kernel of the $\rot$ operator is characterized by the Powell-Sabin $C^{1}$ element. In a modern language, this means that the vector-valued Lagrange space on this mesh fits in a smoother de~Rham (Stokes) complex \cite{boffi2020convergence}. Nevertheless, the $C^{0}$ continuity may still lead to spurious solutions of the source problem on L-shaped domains \cite[Chapter 5]{arnold2018finite}. Practically, it is crucial to construct efficient bases for high order elements. This is always an art since one has to balance condition numbers, sparsity, fast evaluation, rotational symmetry, etc. See, e.g., \cite{ainsworth2001hierarchic,jorgensen2004higher,schoberl2005high,sun2001construction,webb1999hierarchal,xin2011well,xin2013construction,zaglmayr2006high}. Nodal elements may also be preferred due to the canonical implementation \cite{badia2012nodal,boyse1992nodal}. Finite elements with extra vertex/edge smoothness (supersmoothness) also find applications in solving the Hamilton-Jacobi-Bellman equations and may avoid averaging in semi-Lagrange methods \cite{wouter-thesis}.

This paper investigates another variant of the vector-valued Lagrange finite elements with the following desirable properties: 1) The elements allow discontinuity (not $C^{0}$-conforming), such that spurious solutions arising from regularity should not appear. 2) The elements fit in complexes. 3) There exist Lagrange type nodal bases, and thus the construction (condition number, orthogonality) can be based on results of scalar problems, c.f., \cite{karniadakis2013spectral}.  
The idea of the construction is to allow discontinuity of certain components across the faces of the elements. We will refer to these elements as {\it partially discontinuous nodal elements}.
 In two space dimensions (2D), the Stenberg element \cite{Stenberg2010} satisfies all these conditions. The corresponding complex was investigated in \cite{christiansen2016nodal} (see also \cite{gillette2018nonstandard} for the cubical case). Nevertheless, the three dimensional (3D) case was left open in \cite{christiansen2016nodal} (see Section 2.2 for a detailed review).

This paper aims to construct the 3D version of the partially discontinuous nodal finite elements. In particular, we obtain two finite element de~Rham complexes: one contains an $H(\div)$-conforming partially discontinuous nodal element, and another contains an $H(\curl)$ version (note that they cannot fit in the same complex). The $H(\div)$   element was constructed in \cite{christiansen2016nodal}. To the best of our knowledge, the complex constructed herein is new.   Finite elements for $H(\curl)$ with a higher order supersmoothness can be found in \cite{christiansen2016nodal}. However, they do not admit Lagrange type bases. Complexes containing vector-valued Lagrange elements on special meshes were constructed in \cite{fu2020exact,guzman2020exact}, primally in the context of fluid problems. Nevertheless, the $C^{0}$ continuity does not exclude spurious solutions for the Maxwell source problems. In this paper, we obtain a new  partially discontinuous nodal $H(\curl)$ finite element and the corresponding complex on the Worsey-Farin split of tetrahedra.  

Breaking some face degrees of freedom (e.g., treating the degrees of freedom of the tangential components as interior ones by allowing multi-values on neighboring elements) is equivalent to enriching face-based bubble functions. The idea of using matrix-valued Lagrange spaces enriched with $H(\div)$ symmetric matrix bubbles led to partially discontinuous nodal elements for the Hellinger-Reissner principle \cite{Hu2015,hu2015family,hu2016finite},  
see \cite{chen2021geometric} for an extension of this idea.

The rest of this paper will be organized as follows. In Section 2, we introduce notation and review existing results. In Section 3 we present a new $H(\curl)$ element and the corresponding complex. In Section 4, we construct a discrete de Rham complex which contains the nodal $H(\div)$ elements \cite{christiansen2016nodal}. In Section 5,   we construct well-conditioned bases for both 2D and 3D  partially discontinuous nodal elemets. We end the paper with concluding remarks.

\section{Preliminaries and existing results}

\subsection{Preliminaries}\label{notation}We assume that $\Omega$ is a polygonal or polyhedral domain. 
Given a mesh, we will use $\mathcal{V}$ to denote the set of vertices, $\mathcal{E}$ for edges, $\mathcal{F}$ for faces, and $\mathcal{T}$ for the 3D simplices. We use $V$, $E$, $F$, and $T$ to denote the number of vertices, edges, faces, and tetrahedra, respectively. From Euler's formula, one has $V-E+F=1$ in 2D and $V-E+F-T=1$ in 3D for contractible domains. 

For a triangle $f$, we denote its Clough-Tocher split as $f_{\CT}$, which is the collection of three triangles obtained by connecting the vertices to an interior point of $f$.
For a tetrahedron $K$ in $\mathcal T$,  the Worsey-Farin split of $K$, denoted by $K_{\WF}$, is obtained by adding a vertex to the interior of $K$, connecting this vertex to the vertices of $K$; then adding an interior vertex to each face of $K$, and connecting this point to the vertices of the face and the interior point (see Figure \ref{fig:splits}). Therefore, each face is divided into three triangles like the Clough-Tocher split. The set of the Worsy-Farin splits of all the tetrahedra in $\mathcal T$ is referred to as the Worsey-Farin mesh and denoted by $\mathcal T_{\WF}$. We refer to \cite{fu2020exact,guzman2020exact,lai2007spline} for more details.

We use $\bm{\nu}_{f}$ and $\bm{\tau}_{f}$ to denote the unit normal and tangential vectors of a simplex $f$, respectively. In 2D, the tangential and normal directions of an edge are uniquely defined up to an orientation. For edges in 3D there are one unit tangential and two linearly independent unit normal directions, and for faces in 3D there are one unit normal and two linearly independent unit tangential directions. We will write $\bm{\tau}_{e}$,  $\bm{\nu}_{e, i}$ and $\bm{\nu}_{f}$,  $\bm{\tau}_{f, i}$, $i=1, 2$ for these cases, respectively.

In our discussions, $C^{r}(\mathcal{V})$ includes piecewise smooth functions for which derivatives up to order $r$ coincide at vertices. 
Similarly we can define $C^{r}(\mathcal{E})$ and $C^{r}(\mathcal{F})$, $C^{r}(\mathcal{T})$ for the space of functions with $C^r$ continuity on the edges, faces, and 3D cells, repectively.

In $d$D ($d=2,3$), we define the gradient operator $\grad=\nabla$ and the $\div$ operator $\div=\nabla\cdot$.
In 3D, for $\bm u=(u_1,u_2,u_3)^{\mathrm T}$, we define
\[\curl \bm u=(\partial_{x_2}u_3-\partial_{x_3}u_2, \partial_{x_3}u_1-\partial_{x_1}u_3, \partial_{x_1}u_2-\partial_{x_2}u_1)^{\mathrm T}.\]
In 2D, we define
\[\rot \bm u= \partial_{x_1}u_2-\partial_{x_2}u_1\text{ for }\bm u=(u_1,u_2)^{\mathrm T}\text{ and }\curl u=(\partial_{x_2}u,-\partial_{x_1}u) \text{ for a scalar }u.\]

We review some basic facts about the de Rham complex; further details can be found, for instance, in \cite{Arnold.D;Falk.R;Winther.R.2006a,bott2013differential}. 
The de Rham complex consists of differential forms and exterior derivatives. In 3D, the de Rham complex on $\Omega$ reads
\begin{equation}\label{general-complex}
\begin{tikzcd}
0 \arrow{r} & \mathbb R\arrow{r}{\subset}& H^{1}(\Omega) \arrow{r}{\grad} &H(\curl;\Omega)\arrow{r}{\curl} & H(\div;\Omega)   \arrow{r}{\div} &L^2(\Omega)  \arrow{r}{} & 0,
 \end{tikzcd}
\end{equation}
with $\curl\grad=\div\curl=0$. A complex is called exact at a space if the kernel of the subsequent operator is identical to the range of the previous one. For example,   we say that the complex \eqref{general-complex} is exact at $H(\curl;\Omega)$ if $\ker(\curl,H(\curl))=\ran(\grad,H^{1})$. Here for a linear operator $d: X\to Y$, $\ker(d, X)\subset X$ denotes the kernel of $d$ and $\ran(d, Y)\subset Y$ denotes the range. 
The complex \eqref{general-complex} is called exact if it is exact at all the spaces.

If there are finite element spaces $\Sigma_h\subset H^1(\Omega)$, $V_h\subset H(\curl;\Omega)$, $X_h\subset H(\div;\Omega)$, and $W_h\subset L^2(\Omega)$ such that
\begin{equation}\label{discrete-complex}
\begin{tikzcd}
0 \arrow{r} & \mathbb R\arrow{r}{\subset}& \Sigma_h\arrow{r}{\grad} &V_h\arrow{r}{\curl} & X_h   \arrow{r}{\div} &W_h  \arrow{r}{} & 0
 \end{tikzcd}
\end{equation}  
is a complex, we say that \eqref{discrete-complex} is a finite element subcomplex of \eqref{general-complex}. In this case, a necessary   condition for the exactness  is the  following dimension condition
\begin{align}\label{dimcount}
\dim(\mathbb R)-\dim \Sigma_h+\dim V_h-\dim X_h+\dim W_h=0.
\end{align}

As an example, the Lagrange element space, N\'{e}d\'{e}lec element space of the first kind \cite{Nedelec.J.1980a}, Raviart-Thomas (RT) element space \cite{Nedelec.J.1980a}, and discontinuous element space fit into a discrete de Rham complex. The Lagrange element space, N\'{e}d\'{e}lec element space of the second kind \cite{Nedelec.J.1986a}, Brezzi-Douglas-Marini (BDM) element space \cite{Nedelec.J.1986a}, and discontinuous element space also fit into a discrete de Rham complex. For a detailed discussion of these finite element spaces, see, e.g., \cite{boffi2013mixed}.

We use $\mathcal{P}_p$  to denote the polynomial space of degree $p$ (a nonnegative integer). Denote
 \[\NED_p^1=(\mathcal P_{p-1})^3 + S_p\text{ with }S_p=\{ q\in (\mathcal P_{p})^3:q\cdot \bm x=0\}\] 
as the shape function spaces of the first-kind N\'{e}d\'{e}lec element with polynomial degree $p$, and 
 \[\RT_p=(\mathcal P_{p-1})^3 + \bm x\mathcal P_{p-1},\] 
as the shape function space of the RT element with the polynomial degree $p$. The shape function space of the second-kind N\'{e}d\'{e}lec element and the BDM element with the polynomial degree $p$ are 
\[\NED_p^2=[\mathcal{P}_p]^3\text{ and } \BDM_p=[\mathcal{P}_p]^3,\]
respectively.
We denote the N\'{e}d\'{e}lec finite element space of the first and second kind with the polynomial degree $p$ on mesh $\mathcal T$ as $\NED_p^1(\mathcal T)$ and $\NED_p^2(\mathcal T)$, respectively.

In 2D, the de Rham complex reads
\begin{equation*}
\begin{tikzcd}
0 \arrow{r} & \mathbb R\arrow{r}{\subset}& H^{1}(\Omega) \arrow{r}{\grad} &H(\rot;\Omega)\arrow{r}{\rot} & L^2(\Omega)  \arrow{r}{} & 0,
 \end{tikzcd}
\end{equation*}
or
\begin{equation*}
\begin{tikzcd}
0 \arrow{r} & \mathbb R\arrow{r}{\subset}& H^{1}(\Omega) \arrow{r}{\curl} &H(\div;\Omega)\arrow{r}{\div} & L^2(\Omega)  \arrow{r}{} & 0.
 \end{tikzcd}
\end{equation*}

 {
\begin{center}
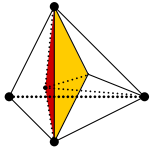
\begin{figure}
\setlength{\unitlength}{1.2cm}
\begin{picture}(2,2)(5,0)

\put(6,0){
\begin{picture}(2,2)


{
\color{yellow}\polygon*(0, 0)(0, 1.5)(0.368, 0.748)
\color{red}\polygon*(0, 0)(0, 1.5)(-0.1, 0.6)
}

\put(-0.5, 0.5){\circle*{0.1}}
\put(1, 0.5){\circle*{0.1}}
\put(0, 1.5){\circle*{0.1}}
\put(0, 0){\circle*{0.1}}
\put(0, 0){\line(-1, 1){0.5}}
\put(0, 0){\line(2, 1){1}}

\put(1, 0.5){\line(-1, 1){1}}
\put(-0.5, 0.5){\line(1, 2){0.5}}
\put(0, 0){\line(0, 1){1.5}}
\put(-0.1, 0.6){\circle*{0.05}}

\multiput(-0.5,0.5)(0.05,0){30}{\circle*{0.03}}

\put(0, 0){\line(1, 2){0.38}}
\put(0, 1.5){\line(1, -2){0.38}}
\put(1, 0.5){\line(-5, 2){0.62}}
\multiput(-0.1, 0.6)(0.05,0.015){10}{\circle*{0.02}}
\multiput(-0.1, 0.6)(0.005, 0.05){18}{\circle*{0.02}}
\multiput(-0.1, 0.6)(0.008, -0.05){12}{\circle*{0.02}}
\multiput(-0.1, 0.6)(0.05, -0.004){20}{\circle*{0.02}}
\end{picture}

}

\end{picture}
\caption{The Worsey-Farin split of a tetrahedron (relative to one face)}
 \label{fig:splits}
\end{figure}
\end{center}
}
\subsection{Existing results} 
In this section, we briefly review the constructions in 2D in \cite{christiansen2016nodal} and some questions that were left open there. 

In 2D, the Stenberg $H(\div)$ element can be obtained by breaking the tangential component of a vector-valued  Lagrange element on each interior edge of the mesh, i.e., the tangential trace of the functions in the Stenberg $H(\div)$ element space from two sides of an edge can be different. Therefore, the continuity of the Stenberg element is in between the corresponding vector-valued $C^{0}$ Lagrange element and the discontinuous element (see Figure \ref{fig:cdg}). The Stenberg element fits in a discrete de~Rham complex in 2D, which begins with the Hermite element taking vertex derivatives as degrees of freedom \cite{christiansen2016nodal}.

\begin{figure}[htbp]
\begin{center}
\includegraphics[width=0.68\linewidth]{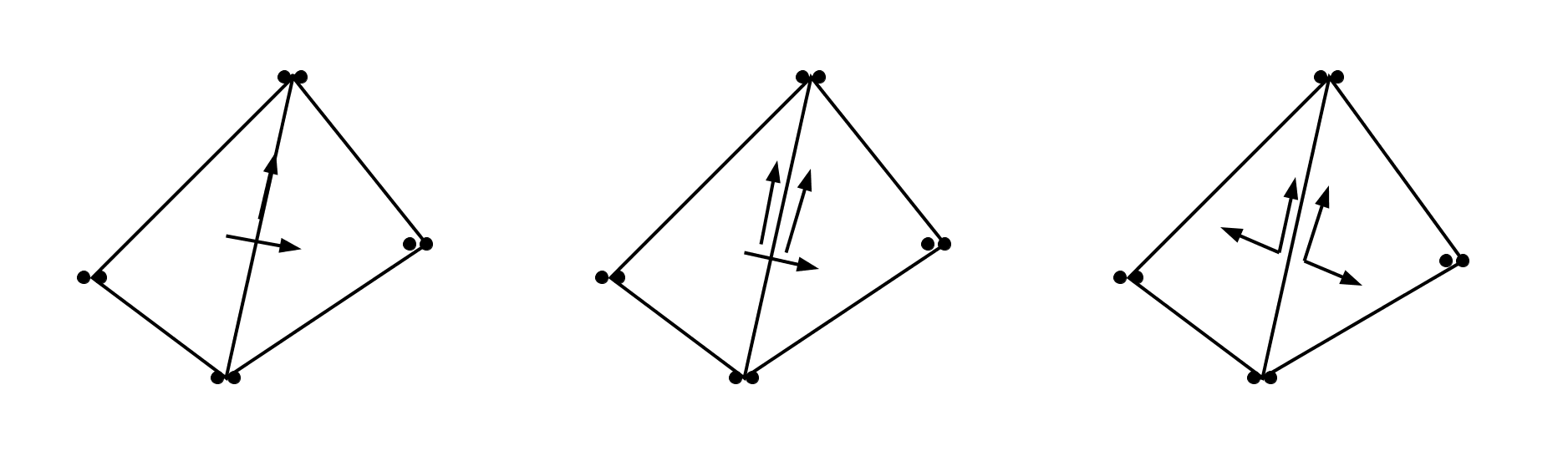}
\caption{Continuous, partially discontinuous elements, and discontinuous. Left: vector-valued Lagrange elements. Both the tangential and the normal components are continuous across edges. Middle: partially discontinuous nodal elements for $H(\div)$ (Stenberg element). The normal component is continuous while the tangential component is discontinuous in general. Right: discontinuous elements. Both tangential and normal components are discontinuous in general.}
\label{fig:cdg}
\end{center}
\end{figure}

The 2D  Stenberg $H(\div)$   element  can be extended to 3D (see \cite[Section 3.5]{christiansen2016nodal} and $\mathcal{W}_{p+1}^{2}$ in \eqref{complex-HZ}). However, the complex containing it was not discussed in \cite{christiansen2016nodal}.

The construction in \cite{christiansen2016nodal} contains two complexes in 3D: one begins with the Hermite element, and another begins with a $C^{0}$ element with second order supersmoothness at the vertices (formally, $C^{2}(\mathcal{V})$, although the function is not $C^{2}$). The latter can be viewed as a 3D extension of the Falk-Neilan complex for the Stokes problem \cite{falk2013stokes} in terms of supersmoothness. The $H(\curl)$ element in the former complex does not admit nodal basis, while the $H(\curl)$ element in the latter complex can be represented in terms of scalar Hermite bases. Nevertheless, it was left open in \cite{christiansen2016nodal} to construct a sequence containing an $H(\curl)$ element with Lagrange bases.
We solve these problems in this paper.

\section{$H(\curl)$ nodal elements in three space dimensions}\label{sec:curl}

\subsection{Motivation}
A natural candidate for a partially discontinuous nodal $H(\curl)$ element is obtained by breaking the normal components of the vector-valued Lagrange  finite element on each face (at the same time, the $C^{0}$ continuity at vertices and on edges is retained):
$$
Z_{h}^{p}:= \NED_p^2(\mathcal T)\cap C^{0}(\mathcal{V})\cap C^{0}(\mathcal{E}).
$$
 This element is in between the $C^{0}$ vector-valued Lagrange element and the N\'ed\'elec element in terms of continuity. Nevertheless, this element unlikely fits in a finite element de Rham complex due to the following observation. In most existing cases, the ``restriction'' (precise definition discussed in \cite{christiansen2016generalized}) of a higher-dimensional finite element (complex) to a lower-dimensional cell is a lower-dimensional finite element (complex). This holds for the standard finite element differential forms \cite{arnold2018finite}, as well as all the examples in \cite{christiansen2016nodal}. Therefore, to construct a complex in 3D, we want the restrictions of the involved spaces and the degrees of freedom on each face to form a 2D complex as well. Nevertheless, this is unlikely for the case of $Z_{h}^{p}$.
Because the restriction of $Z_{h}^{p}$ on each face is a 2D Lagrange element, which unlikely fits in a de~Rham complex, as reflected in the deficiency of $Z_{h}^{p}$ (at least for low order cases) as a discretization for the Stokes problem, c.f.,  \cite{boffi2013mixed,guzman2017scott,scott1985norm}. 

Although the above observation already suggests that $Z_{h}^{p}$ may not be a good candidate for an $H(\curl)$ nodal element, one may still ask if $Z_{h}^{p}$, as a space in between the vector-valued Lagrange element (which produces spurious solutions) and the N\'ed\'elec element (which excludes spurious solutions), suffers from spurious solutions for eigenvalue problems. To the best of our knowledge, the emergence of spurious solutions for eigenvalue problems is not predicted by theory in most cases. Therefore we answer this question with numerical tests.

We apply the $H(\rot)$ version of the Stenberg element in 2D defined by 
$$
Y_{h}^{p}:=\NED_p^2(\mathcal T)\cap C^{0}(\mathcal{V}),
$$
 and $Z_{h}^{p}$ in 3D to solve the following Maxwell eigenvalue problem on the domain $\Omega=(0,\pi)^d$, $d=2,3$:
\begin{equation}\label{maxwell}
	\begin{aligned}
		\curl\curl \bm u=\lambda \bm u \ \text{in } \Omega,\\
		\div\bm u=0 \text{ in }\Omega,\\
		\bm u\times\bm n =0 \text{ on }\partial \Omega.
	\end{aligned}
\end{equation}
In 2D, the eigenvalues of this problems are $m^2+n^2$ \cite{arnold2018finite}, and the corresponding eigenvectors are 
\[\bm u(x_1,x_2)=\left(\begin{matrix}
	n\sin(mx_1)\cos(nx_2)\\
	-m\cos(mx_1)\sin(nx_2)
\end{matrix}\right),\]
where $m,n$ are nonnegative integers, not both 0.
In 3D, the eigenvalues of this problems are $m^2+n^2+l^2$, and the corresponding eigenvectors are
\[\bm u(x_1,x_2,x_3)=\left(\begin{matrix}
	(l-n)\cos(mx_1)\sin(lx_3)\sin(nx_2)\\
	(m-l)\cos(nx_2)\sin(mx_1)\sin(lx_3)\\
	(n-m)\cos(lx_3)\sin(mx_1)\sin(nx_2)
	\end{matrix}\right),\] 
where $m,n,l$ are nonnegative integers, not all 0.

We first apply the Stenberg $H(\rot)$ finite element space $Y_{h}^{p}$ to solve the problem \eqref{maxwell} in 2D. The numerical results are shown in Table \ref{tab5}. 
\begin{table}[h]
	\centering
	\caption{Numerical eigenvalues for 2D \eqref{maxwell} by the 2nd-order vector-valued Lagrange element  and the Stenberg $H(\rot)$ element } \label{tab5}
	\begin{tabular}{ccccccccccccc} 
		\hline
		exact&1&1&2&4&4&5&5&8&9&9\\
		\hline
		$H^1$ & 1.0000 &1.0000& 1.1375&2.0000 &2.5668& 3.1367&4.0000&4.0000&4.1947&5.0000\\
		$H(\rot)$& 1.0000 &1.0000&2.0000&4.0000 &4.0000& 5.0000&5.0000&8.0000&9.0001&9.0001\\
		\hline
	\end{tabular}
\end{table}
From the table, we see that spurious solutions do not appear. Therefore the vector-valued Lagrange elements are remedied by breaking the normal degrees of freedom on edges.
Then   we use $Z_{h}^{p}$ to solve the problem \eqref{maxwell} for $d=3$. Table \ref{tab2} shows the numerical  eigenvalues around 3. In this case, spurious solutions occur. Therefore relaxing the normal continuity on faces (equivalently, enriching face-based bubbles) on the vector-valued Lagrange element is not enough to give a proper discretization of the Maxwell equations.
\begin{table}[h]
	\centering
	\caption{Numerical eigenvalues for 3D \eqref{maxwell} by the 3rd-order vector-valued Lagrange elements with discontinuous normal components} \label{tab2}
	\begin{tabular}{ccccccccccccc} 
		\hline
		exact&1&1&1&2&2&2&3&4&4&4\\ \hline
		numerical&2.9924  &  2.9960   & 2.9972   & 3.0000  &  3.0000&    3.0042  &  3.0057  &  3.0059 &   3.0069 &   3.0158\\
		\hline
		
		\hline
	\end{tabular}
\end{table}

Now both heuristic argument and numerical experiments inspire us to seek a different construction other than $Z_{h}^{p}$. We cannot further relax the vertex and edge continuity, which is required to ensure the existence of nodal bases. Motivated by the heuristic argument on the face modes, now we plan to construct an element such that its restriction on each face is a proper velocity finite element for the Stokes problem. Recall that the $\mathcal{P}_{2}$-$\mathcal{P}_{1}$ Scott-Vogelius pair on the Clough-Tocher split of triangular meshes (see Figure \ref{fig:CT}) leads to a complex and is thus stable  \cite{arnold1992quadratic,neilan2020stokes}. This motivates us to use the Worsey-Farin split of a tetrahedron \cite{lai2007spline,worsey1987ann}, which, restricted on each face of the macro tetrahedron, is a Clough-Tocher split.

\begin{center}
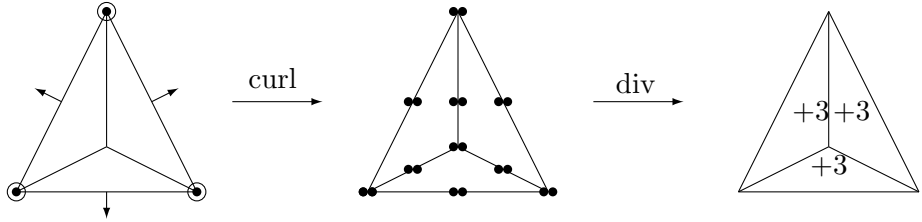
\begin{figure}
\setlength{\unitlength}{1.2cm}
\begin{picture}(2,2)(3,0)
\put(0,0){
\begin{picture}(2,2)
\put(-1, 0){\line(1,2){1}} 
\put(0, 2){\line(1,-2){1}}
\put(-1,0){\line(1,0){2}}
\put(-1.,0){\circle*{0.1}}
\put(-1.,0){\circle{0.2}}
\put(1.,0){\circle*{0.1}}
\put(1.,0){\circle{0.2}}
\put(0,2){\circle*{0.1}}
\put(0,2){\circle{0.2}}
\put(0, 0.5){\line(0, 1){1.5}}
\put(0, 0.5){\line(2, -1){1}}
\put(0, 0.5){\line(-2, -1){1}}
\put(0.5, 1){\vector(2,1){0.3}}
\put(-0.5, 1){\vector(-2,1){0.3}}
\put(0, 0){\vector(0,-1){0.3}}
\end{picture}
}

\put(1.5, 1){\vector(1, 0){1}}
\put(1.68, 1.15){$\curl$}

\put(4,0){
\put(-1, 0){\line(1,2){1}} 
\put(0, 2){\line(1,-2){1}}
\put(-1,0){\line(1,0){2}}
\put(-0.05,2){\circle*{0.1}}
\put(0.05,2){\circle*{0.1}}
\put(-1.05,0){\circle*{0.1}}
\put(-0.95,0){\circle*{0.1}}
\put(-1.05,0){\circle*{0.1}}
\put(-0.95,0){\circle*{0.1}}
\put(1.05,0){\circle*{0.1}}
\put(0.95,0){\circle*{0.1}}
\put(0.55, 1){\circle*{0.1}}
\put(0.45, 1){\circle*{0.1}}
\put(-0.55, 1){\circle*{0.1}}
\put(-0.45, 1){\circle*{0.1}}

\put(-0.05, 1){\circle*{0.1}}
\put(0.05, 1){\circle*{0.1}}
\put(-0.05, 0){\circle*{0.1}}
\put(0.05, 0){\circle*{0.1}}
\put(0, 0.5){\line(0, 1){1.5}}
\put(0, 0.5){\line(2, -1){1}}
\put(0, 0.5){\line(-2, -1){1}}
\put(-0.45, 0.25){\circle*{0.1}}
\put(-0.55, 0.25){\circle*{0.1}}
\put(0.45, 0.25){\circle*{0.1}}
\put(0.55, 0.25){\circle*{0.1}}
\put(0.05, 0.5){\circle*{0.1}}
\put(-0.05, 0.5){\circle*{0.1}}
}

\put(5.5, 1){\vector(1, 0){1}}
\put(5.75, 1.1){{$\div$}}

\put(8,0){
\begin{picture}(2,2)
\put(-1, 0){\line(1,2){1}} 
\put(0, 2){\line(1,-2){1}}
\put(-1,0){\line(1,0){2}}
\put(0, 0.5){\line(0, 1){1.5}}
\put(0, 0.5){\line(2, -1){1}}
\put(0, 0.5){\line(-2, -1){1}}
\put(-0.2, 0.2){+3}
\put(-0.4, 0.8){+3}
\put(0.05, 0.8){+3}
\end{picture}
}

\end{picture}
\caption{Clough-Tocher complex with discontinuous pressure.\newline
The figure shows the lowest order case: the first space is piecewise cubic, the second is continuous piecewise quadratic and the third is piecewise linear. }
\label{fig:CT}
\end{figure}
\end{center}

\begin{figure}[htbp]

\includegraphics[width=0.68\linewidth]{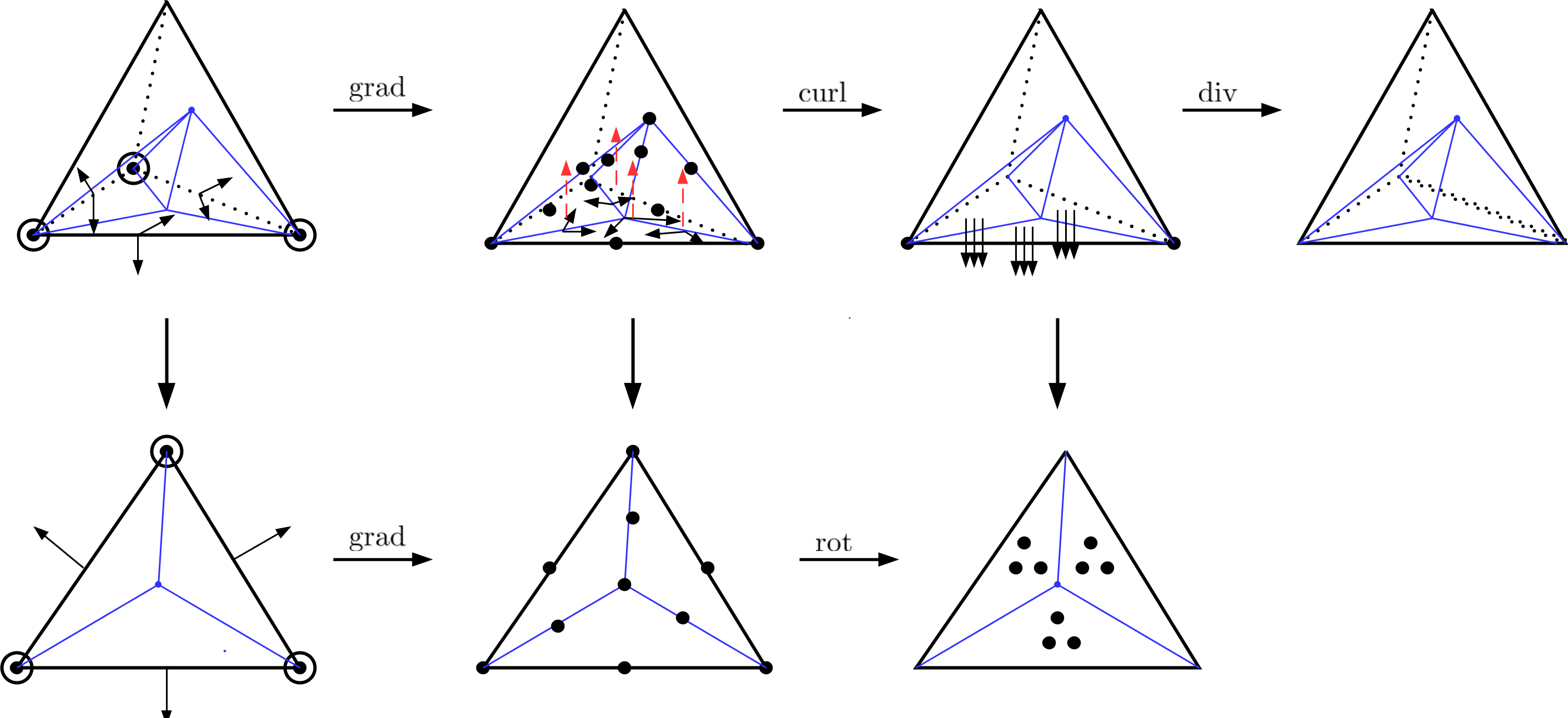}
\caption{The complex containing the 3D $H(\curl)$ nodal element and its 2D restriction. For $H(\rot)$/$H(\curl)$/$H(\div)$ elements, a dot $\bullet$ represents the evaluation of the vector-valued function, which counts as two/three degrees of freedom. In the 3D $H(\curl)$ element, the four red dashed lines indicate the normal components associated with the four Lagrange points on the face. They are counted as interior degrees of freedom and do not yield inter-element continuity.}
\label{fig:curl}
\end{figure}

\subsection{Construction}
We consider a complex  for $p\geq 3$ as follows:
\begin{equation}\label{seq:curl}
\begin{tikzcd}
0  \arrow{r}& \mathbb{R} \arrow{r}{\subset} &\mathcal V_p^0(\mathcal T)\arrow{r}{\grad} & \mathcal V_{p-1}^1(\mathcal T) \arrow{r}{\curl} &\mathcal V_{p-2}^2(\mathcal T_{\WF})\arrow{r}{\div} & \mathcal V_{p-3}^3(\mathcal T_{\WF})\arrow{r} & 0,
\end{tikzcd}
\end{equation}
where each space is constructed on the Worsey-Farin split. The Worsey-Farin $C^{1}$ element \cite{worsey1987ann} requires a mesh alignment, the subdivision point on each face of the macro tetrahedron lies on the line connecting the subdivision points of the two macro tetrahedra sharing the face. Nevertheless, such a condition is not required in the construction below.

\paragraph{{Space $\mathcal V_p^0(\mathcal T)$.}}

The scalar space $\mathcal V_p^0(\mathcal T)$ is a high order generalization of the Worsey-Farin macroelement from \cite{worsey1987ann}. The shape function space on  macroelement $K$ consists of $C^1$ piecewise polynomials of degree $p$ on the Worsey-Farin split, i.e.
\[\mathcal V_p^0(K)=\{u\in C^1(K): u|_{K_i}\in \mathcal P_p(K_i),\text{ for } K_i\in K_{\WF}\text{ with }\cup_{i=1}^4K_i=K\}.\]
 Across the macroelements, the functions in  $\mathcal V_p^0(\mathcal T)$ are of $C^{1}$ continuity at vertices and edges, and of $C^{0}$ continuity across faces, that is
\[\mathcal V_p^0(\mathcal T)=\{u\in H^1(\Omega): u\in C^1(\mathcal E)\cap C^1(\mathcal V), \ u|_K\in \mathcal V_p^0(K)\text{ for }  K\in\mathcal T\}.\]

For the lowest order case ($p=3$, c.f., \cite{lai2007spline}), if the mesh alignment condition holds, then the functions are also of $C^{1}$ continuity across the macroelements. The subsequent discussions do not rely on such $C^{1}$ smoothness across macroelements, therefore the mesh alignments are not required. 

For $u\in \mathcal V_p^0(\mathcal T)$, the degrees of freedom can be given by
\begin{itemize}
\item
function value and the first order derivatives of $u$  at each vertex $\bm x\in \mathcal{V}$,
$$
u(\bm x), \quad \partial_{i}u(\bm x), \quad i=1, 2, 3;
$$
\item
moments of $u$ on each edge $e\in \mathcal E$,
$$
\int_{e}u\cdot q\, \mathrm{d}\bm x, \quad q\in \mathcal{P}_{p-4}(e);
$$
\item
moments of the normal derivatives of $u$ on each edge $e\in \mathcal E$,
$$
\int_{e}\partial_{\bm{\nu}_{e,1}}u\cdot q_{1}\, \mathrm{d}\bm x,\quad \int_{e}\partial_{\bm{\nu}_{e,2}}u\cdot q_{2}\, \mathrm{d}\bm x, \quad q_{1}, q_{2}\in \mathcal{P}_{p-3}(e);
$$
\item
moments of $\grad_f u$ on each face $f\in \mathcal F$,
$$
\int_{f} \grad_fu\cdot q\, \mathrm d \bm x, \quad q\in \grad B_{p}({f}_{\mathrm{CT}}),
$$
where the space $B_p(f_{\mathrm{CT}}):=\left \{w\in C^1(f): w|_{f_i}\in \mathcal{P}_{p} \text{ for all } f_i\in f_{\CT}\text{ with }\right.$\\$\left.\cup_{i=1}^3 f_i=f\text{ and } u=\partial_n u=0 \mbox{ on } \partial f\right \}$ consists of the bubble functions of the 2D Clough-Tocher element on the face subdivision $f_{\mathrm{CT}}$, see the first element in Figure \ref{fig:CT};
\item 
interior degrees of freedom on $K\in\mathcal T$,
$$
\int_{K} u q\, \mathrm d\bm x, \quad q\in B_{p}({K}_{\mathrm{WF}}),
$$
where $B_{p}(K_{\mathrm{WF}}):=\left \{w\in C^1(K): w|_{K_i}\in \mathcal{P}_{p} \text{ for all } K_i\in K_{\WF} \text{ with }\right.$ \\$\left.\cup_{i=1}^4K_i=K \text{ and } u=0, \grad u=0 \mbox{ on } \partial K\right \}$.
\end{itemize}

The dimension of the shape function space $\mathcal V_p^0(K)$ has been given in \cite[Theorem 3.5]{schenck2018subdivision} with a homological approach, that is, 
\begin{align}\label{total-dim}
{p+3 \choose 3} +3{p-1 \choose 3}+8{p\choose 3}= 2(p^{3}-3p^{2}+5p-1).
\end{align}
Note that the above degrees of freedom at vertices, edges, and faces are linearly independent since they are a part of the degrees of freedom for defining the $C^1$ element in \cite{guzman2020exact}. Therefore the unisolvence holds and the dimension of the bubble space (number of interior degrees of freedom) is equal to the total dimension of \eqref{total-dim} subtracting the number of degrees of freedom at vertices,  edges, and faces. The number of face degrees of freedom, i.e., the dimension of $B_{p}(f_{\mathrm{CT}})$ can also be derived in a similar way based on the dimension formula of the 2D Clough-Tocher macroelement, c.f., \cite{christiansen2016generalized,schenck2018subdivision}. 

Therefore, we obtain the following dimension count of $\mathcal V_p^0(\mathcal T)$:
 $$
 4V+[(p-3)+2(p-2)]E+[3/2(p^{2}-p+2)-(6p-6)]F+2(p-1)(p-2)(p-3)T.
$$

\begin{remark}\label{rmk:supersmoothness}
As a $C^{1}$ tetrahedral spline, the shape function space of $\mathcal V_p^0(K)$ has been studied with the Bernstein-B\'{e}zier technique (see \cite{lai2007spline} for example).  In \cite{lai2007spline}, this space is defined with the $C^{2}$ continuity at the interior subdivision point. Nevertheless, the $C^{2}$ condition is intrinsic supersmoothness \cite{floater2018super}, meaning that it can be removed from the definition and the $C^{1}$ continuity would automatically imply the $C^{2}$ condition at the subdivision point. We refer to \cite{floater2018super} and the references therein for more details on supersmoothnesses. 
\end{remark}

\paragraph{Space $\mathcal{V}_{p}^{1}(\mathcal T)$.}
The shape function space on macroelement $K$ is define as 
\[\mathcal V_p^1(K)=\{\bm u\in C^0(K):\bm u|_{K_i}\in \mathcal P_p,\ \forall K_i\in K_{\WF}\text{ with } \cup_{i=1}^4K_i=K\}.\]
Then the space $\mathcal{V}_{p}^{1}(\mathcal T)$ is defined as 
\[\mathcal V_p^1(\mathcal T)=\{\bm u\in H(\curl;\Omega): \bm u\in C^0(\mathcal E)\cap C^0(\mathcal V),\ \bm u|_K\in \mathcal V_p^1(K), \forall K\in \mathcal T\}\]

The degrees of freedom are given by
\begin{itemize}
\item
function value of each component $u_i$ of $\bm u$ at each vertex $\bm x\in \mathcal{V}$,
$$
u_{i}(\bm x), \quad i=1, 2, 3;
$$
\item
moments of each component $u_i$ of $\bm u$ on each edge $e\in\mathcal E$,
$$
\int_{e}u_{i}\cdot q\, \mathrm{d}\bm x, \quad q\in \mathcal{P}_{p-2}(e), i=1, 2, 3;
$$
\item
moments of the tangential components of $\bm u$ on each face $f\in\mathcal F$,
$$
\int_{f}(\bm{u}\times \bm{\nu}_f)\cdot \bm{q}\, \mathrm{d}\bm x, \quad q\in \left [B^0_{p}(f_{\CT})\right ]^{2}, 
$$
where $B_p^0(f_{\mathrm{CT}}):=\left \{w\in C^{0}(f): w|_{f_i}\in \mathcal{P}_{p} \text{ for all } f_i\in f_{\CT} \text{ with }\cup_{i=1}^3f_i=f\right.$\\ $\left.\text{and } u=0 \mbox{ on } \partial f\right\}$ and $\bm{u}\times \bm{\nu}_f$ is considered as a 2D vector field;
\item 
interior degrees of freedom.
\end{itemize}
The interior degrees of freedom consists of function values at $3[1+8(p-1)+9(p-1)(p-2)+2(p-1)(p-2)(p-3)]=6p^3-9p^2+9p-3$ interior Lagrange points of each $K\in\mathcal T$
and the values of the normal component at $1+3(p-1)+3/2(p-1)(p-2)=(3p^2-3p+2)/2$ Lagrange points on each face of  $K$.
Therefore, the dimension of $\mathcal{V}_{p}^{1}(\mathcal T)$ is 
$$
3V+3(p-1)E+(3p^{2}-3p+2)F+(6p^3-3p^2+3p+1)T.
$$

\paragraph{Space $\mathcal{V}_{p}^{2}(\mathcal T_{\WF})$.}

From this step, the complex branches into the standard finite element de Rham complex of discrete differential forms \cite{Arnold.D;Falk.R;Winther.R.2006a} on the Worsey-Farin mesh. 

More precisely, the space $\mathcal{V}_{p}^{2}(\mathcal T_{\WF})$ is the BDM element on the Worsey-Farin mesh of the domain $\Omega$, i.e.,
\[\mathcal V_p^2(\mathcal T_{\WF})=\{u\in H(\div;\Omega): u|_{K_i}\in \mathcal P_p,\text{ for } K_i\in K_{\WF}\text{ and } \cup_{i=1}^4 K_i=K\in\mathcal T\}.\]
Since each macro tetrahedron $K$ is divided into $12$ micro tetrahedra, there are $18$ micro faces in the interior of $K$ and $3$ micro faces on each face of $K$. As a result, it follows from the definition of the BDM element that 
the dimension count of $\mathcal{V}_{p}^{2}(\mathcal T_{\WF})$ is $$3/2(p+2)(p+1)F+(9(p+1)(p+2)+6(p+2)(p+1)(p-1))T.$$

\paragraph{Space $\mathcal{V}_{p}^{3}(\mathcal T_{\WF})$.}

The space  $\mathcal{V}_{p}^{3}(\mathcal T_{\WF})\subset L^2(\Omega)$ consists of piecewise polynomials of degree $p$ on the Worsey-Farin mesh $\mathcal T_{\WF}$ of $\Omega$. 
The dimension of $\mathcal{V}_{p}^{3}(\mathcal T_{\WF})$ is $2(p+3)(p+2)(p+1)T$.

\begin{lemma}
For $p\geq 3$, the sequence \eqref{seq:curl}
is a complex and exact on contractible domains.
\end{lemma}
\begin{proof}
By construction,  \eqref{seq:curl} is a complex. It is sufficient to show that \eqref{seq:curl} is exact on contractible domains.  Since the last two spaces $ \mathcal{V}_{p-2}^{2}(\mathcal T_{\WF})$ and $\mathcal{V}_{p-3}^{3}(\mathcal T_{\WF})$ coincides with the BDM and the DG elements in the standard finite element de Rham complexes on the Worsey-Farin mesh \cite{Arnold.D;Falk.R;Winther.R.2006a}, the operator  $\div$ is onto.  This verifies the exactness at the space $\mathcal{V}_{p-3}^{3}(\mathcal T_{\WF})$.  The kernel of the operator $\grad$ is constants, which shows the exactness at the space $\mathcal V_p^0(\mathcal T)$. 
By definition, 
$$
\mathcal{V}_{p-1}^{1}(\mathcal T)=\NED_{p-1}^2(\mathcal T_{\WF})\cap C^{0}(\mathcal{V})\cap C^{0}(\mathcal{E})\cap C^{0}(\mathcal{T}),
$$
where $\NED_{p-1}^2(\mathcal T_{\WF})$ is the second kind N\'ed\'elec finite element space on the Worsey-Farin mesh of $\Omega$. 
Therefore, the kernel space of the operator $\curl$ is as follows
\begin{align}\label{kernel-curl}
\ker\left (\curl, \mathcal{V}_{p-1}^{1}(\mathcal T)\right )=\grad\left [\Sigma_h^p(\mathcal T_{\WF})\cap C^{1}(\mathcal{V})\cap C^{1}(\mathcal{E})\cap C^1(\mathcal T)\right ],
\end{align}
where $\Sigma_h^p(\mathcal T_{\WF})$ is the Lagrange element space of degree $p$ on the Worsey-Farin mesh of $\Omega$. 
The space $\Sigma_h^p(\mathcal T_{\WF})\cap C^{1}(\mathcal{V})\cap C^{1}(\mathcal{E})\cap C^1(\mathcal T)$ coincides with $\mathcal V_p^0(\mathcal T)$, which shows the exactness at $\mathcal{V}_{p-1}^{1}(\mathcal T)$. Then a dimension count in the sense of \eqref{dimcount} proves the exactness at $\mathcal{V}_{p-2}^{2}(\mathcal T_{\WF})$. 

\end{proof}

\section{$H(\div)$ nodal elements in three space dimensions}\label{sec:div}

In this section, we construct a discrete de Rham complex containing the partially discontinuous $H(\div)$ finite element  from \cite[Section 3.5]{christiansen2016nodal}. 
In particular, we will design a complex for $p\geq 2$ as follows
\begin{equation}\label{complex-HZ}
\begin{tikzcd}
0 \arrow{r} &\mathbb{R} \arrow{r}{\subset} & \mathcal{W}_{p+3}^{0}(\mathcal T)   \arrow{r}{\grad} &  \mathcal{W}_{p+2}^{1}(\mathcal T)    \arrow{r}{\curl} &  \mathcal{W}_{p+1}^{2}(\mathcal T)   \arrow{r}{\div} &  \mathcal{W}_{p}^{3}(\mathcal T)   \arrow{r}& 0. 
\end{tikzcd}
\end{equation}

\paragraph{Space $\mathcal{W}^{0}_{p}(\mathcal T)$.}
For $p\geq 4$, the finite element space 
$$
\mathcal{W}^{0}_{p}(\mathcal T):=\left \{u\in H^{1}(\Omega): u\in C^{2}(\mathcal{V})\cap C^{1}(\mathcal{E}), \left . u\right |_{K}\in \mathcal{P}_{p}(K), ~\forall K\in \mathcal{T}\right \},
$$
 coincides with the element $\mathcal{P}_{2, p}\Lambda^{0}(\mathcal{T}^{3})$  in \cite{christiansen2016nodal}, which is also identical to each component of the velocity element in the discrete 3D Stokes complex by Neilan \cite{neilan2015discrete}. The space $\mathcal{W}^{0}_{p}(\mathcal T)$ admits enhanced $C^{2}$ and $C^{1}$  continuities at vertices and edges, respectively, but only $C^{0}$ continuity across faces. Therefore no higher global continuity than $H^{1}$ is imposed.

The global dimension of $\mathcal{W}^{0}_{p}(\mathcal T)$ on the mesh $\mathcal T$ reads
$$
{\dim} \mathcal{W}^{0}_{p} (\mathcal T)=10V+\left [ 2(p-4)+(p-5)  \right ]E+{p-4 \choose 2}  F+{p-1\choose 3} T.
$$

\paragraph{Space $\mathcal{W}^{1}_{p}(\mathcal T)$.} The finite element space 
$$
\mathcal{W}^{1}_{p}(\mathcal T):=\left \{\bm u\in H(\curl;\Omega): \bm u\in C^{1}(\mathcal{V})\cap C^{0}(\mathcal{E}), \left . \bm u\right |_{K}\in \mathcal{P}_{p}(K), ~\forall K\in \mathcal{T}\right \},
$$
 consists of piecewise polynomials of degree less than or equal to $p$ for each component, and the degrees of freedom for $\bm{w}\in \mathcal{W}^{1}_{p}(\mathcal T)$ can be given by
\begin{itemize}
\item
function values and the first order derivatives of $\bm w$ at each vertex $\bm x\in \mathcal{V}$,
$$
\bm{w}(\bm x), \quad \partial_{i}\bm{w}(\bm x), ~i=1,2, 3;
$$
\item
moments of $\bm w$ on each edge $e\in \mathcal{E}$,
$$
\int_{e}\bm{w}\cdot \bm{q}\,\mathrm{d}\bm x, \quad \forall\bm{q}\in \left [\mathcal{P}_{p-4}(e)\right ]^{3};$$
\item
moments of $\curl\bm w$ on each edge $e\in \mathcal{E}$,
$$\quad\int_{e}(\curl\bm{w}\cdot \bm{\nu}_{e,1})\cdot q_{1}\,\mathrm{d}\bm x, \quad \int_{e}(\curl\bm{w}\cdot \bm{\nu}_{e,2})\cdot q_{2}\,\mathrm{d}\bm x,\quad\forall q_{1}, q_{2}\in \mathcal{P}_{p-3}(e),
$$
where $\bm{\nu}_{e,1}$ and $\bm{\nu}_{e,2}$ are two linearly independent  unit normal vectors of edge $e$; 
\item
integrals on each face $f\in \mathcal{F}$,
$$
\int_{f}\left (\bm{w}\times \bm{\nu}_{f}\right )\cdot\curl_{f}\left ( \psi_{f} \lambda_{f, 1}^{2}\lambda_{f, 2}^{2}\lambda_{f, 3}^{2}\right )\mathrm{d}\bm x, \quad \forall\psi_{f}\in \mathcal{P}_{p-5}(f),
$$ 
$$
\int_{f}(\curl \bm{w}\cdot \bm{\nu}_{f})\cdot g\,\mathrm{d}\bm x, \quad \forall g\in \mathcal{P}_{p-4}(f)/\mathbb{R},
$$
where 
$\lambda_{f,1}\lambda_{f,2}\lambda_{f,3}$ is a face bubble associate with face $f$, $\bm \nu_f$ is the unit normal vector of $f$, and
$$\mathcal{P}_{p-4}(f)/\mathbb{R}:=\left \{q\in \mathcal{P}_{p-4}(f): \int_{f}q\, \mathrm{d}\bm x=0\right \};$$ 
\item
interior degrees of freedom in each $K\in \mathcal T$,
\begin{align}\label{interior-dof-curl}
\int_{K}\bm{w}\cdot \bm{z}\,\mathrm{d}\bm x, \quad \forall\bm{z}\in \RT_{p-2}(K).
\end{align}
\end{itemize}

The dimension of $\mathcal W_p^1(\mathcal T)$ is
\begin{align*}
\dim \left ( \mathcal{W}^{1}_{p}(\mathcal T)\right)&=12V+[3(p-3)+2(p-2)]E+[1/2(p-2)(p-3)-1+1/2(p-3)(p-4)]F\\
&\quad\quad+\left(\frac{1}{2}p^{3}-p^{2}-\frac{1}{2}p+1\right)T\\
&=12V+(5p-13)E+(p^{2}-6p+8)F+\left(\frac{1}{2}p^{3}-p^{2}-\frac{1}{2}p+1\right)T.
\end{align*}

\begin{theorem}
The degrees of freedom for $\mathcal{W}^{1}_{p}(\mathcal T)$ are unisolvent and $\mathcal{W}^{1}_{p}$ is $H(\curl)$ conforming.
\end{theorem}
\begin{proof}
On a tetrahedron, the number of the above degree of freedom is 
$$
1/2\left ( p^{3}+6p^{2}+11p+6\right ).
$$
 This coincides with the dimension of $\left [\mathcal{P}_{p}\left (K\right )\right ]^{3}$. Then it suffices to show that if all the degrees of freedom vanish on a function $\bm{w}\in \left [\mathcal{P}_{p}\left (K\right )\right ]^{3}$, then $\bm{w}=0$ on $K\in\mathcal T$. 
 
Define ${\bm{w}_{f}}:=\bm{w}\times \bm{\nu}_{f}$ as a 2D vector field on $f$. The first step is to  show that $\bm{w}_{f}=0$  on each face $f$ of $K$. Actually, from the degrees of freedom, we conclude that
\begin{enumerate}
\item
for each vertex $\bm x\subset f$, it holds that
$
\bm{w}_{f}(\bm x)=\partial_{\bm{\tau}_{f, i}}\bm{w}_{f}(\bm x)=0, ~i=1,2,
$
where $\bm{\tau}_{f, i}$ is the $i$-th tangent vector on $f$;
\item
for each edge $e\subset f$, it holds that
$
\int_{e}\bm{w}_{f}\cdot \bm{q}\,\mathrm{d}\bm x= \int_{e}\div_{f}\bm{w}_{f}\cdot q_{0}\,\mathrm{d}\bm x=0,\ \forall \bm{q}\in \left [\mathcal{P}_{p-4}(e)\right ]^{2}\text{ and } q_{0}\in \mathcal{P}_{p-3}(e);
$
\item 
in the interior of $f$, it holds that
$
\int_{f}\curl_{f}\left ( \psi_{f} \lambda_{f, 1}^{2}\lambda_{f, 2}^{2}\lambda_{f, 3}^{2}\right )\cdot\bm{w}_{f}\,\mathrm{d}\bm x=0,\ \forall \psi_{f}\in \mathcal{P}_{p-5}(f)
$
and
$
\int_{f}\div_{f}\bm{w}_{f}\cdot g\,\mathrm{d}\bm x, 
\ \forall g\in \mathcal{P}_{p-4}(f)/\mathbb{R}.
$
\end{enumerate}
The above functionals in (1) -- (3) are the degrees of freedom of the shape function space $\left [ \mathcal{P}_{p}(f)\right ]^{2}$, see \cite{Arnold2006a}.

To show $\bm{w}_{f}=0$, we first have $\div_{f}\bm{w}_{f}\in \mathcal{P}_{p-1}(f)$ and the vanishing degrees of freedom imply that $\div_{f}\bm{w}_{f}=0$ on $f$. Then it follows that $\bm{w}_{f}=\curl_{f}\phi$ for some $\phi\in \mathcal{P}_{p+1}(f)$. Without loss of generality, we assume that $\phi$ vanishes at one of the vertices.  Now it follows from (1) -- (3) that the following degrees of freedom vanish on $\phi$,
\begin{enumerate}
\item
for each vertex $x\in f$, it holds that
$
\partial_{i}\phi(\bm x)=\partial_{i}\partial_{j}\phi(\bm x)=0 \text{ for }i, j=1, 2
$;
\item
for each edge $e\in \partial f$, it holds that
$
\int_{e}(\curl_{f}\phi)\cdot q\mathrm{d} \bm x, \ \forall q\in [\mathcal{P}_{p-4}(e)]^2$;
\item
in the interior of $f$, it holds that
\begin{align}
\label{face-dof}
	\int_{f}\curl_{f}\phi \cdot  \curl_{f}\left ( \psi_{f} \lambda_{f, 1}^{2}\lambda_{f, 2}^{2}\lambda_{f, 3}^{2}\right )\mathrm{d} \bm x=0, \ \forall \psi_{f}\in \mathcal{P}_{p-5}(f).
\end{align}
\end{enumerate}
This shows that $\curl_{f}\phi=0$ on all the edges of $f$, which further implies that the tangential derivative of $\phi$ along each edge vanishes.  Since $\phi$ vanishes at one vertex of $f$, both $\phi$ and $\curl_{f}\phi$ vanish on $\partial f$. Consequently,
$\phi$ is of the form $\phi=\psi_f\lambda_{f,1}^{2}\lambda_{f,2}^{2}\lambda_{f,3}^{2}$ for some polynomial $\psi_f\in \mathcal{P}_{p-5}(f)$, which together with \eqref{face-dof} yields $\psi_f=0$. This implies  $\phi=0$, and further  $\bm{w}_{f}=\curl_{f} \phi=0$.

Then by the interior degrees of freedom \eqref{interior-dof-curl}, it follows from a similar argument to \cite{Arnold.D;Falk.R;Winther.R.2006a} that $\bm{w}=0$.
The proof for the unisolvence also shows that $\mathcal{W}^{1}_{p}(\mathcal T)$ is $H(\curl)$ conforming. 
\end{proof}

\paragraph{Space $\mathcal{W}^{2}_{p}(\mathcal T)$.} The $H(\div)$ element $\mathcal{W}^{2}_{p}(\mathcal T)$ was defined in \cite{christiansen2016nodal}:
$$
\mathcal{W}^{2}_{p}(\mathcal T):=\left \{\bm{u}\in H(\div): \left . \bm u\right |_{K}\in \mathcal{P}_p (K), \forall K\in \mathcal{T}, \bm{u}\in C^{0}(\mathcal{V}), \bm{u}\cdot\bm{\nu}_{e,i}\in C^{0}(\mathcal{E}), i=1, 2\right \},
$$
where $\bm{\nu}_{e,i},~i=1, 2$ are two linearly independent unit normal vectors of an edge $e$. 

For $u\in \mathcal W_p^2(\mathcal T)$, the degrees of freedom can be given by \begin{itemize}
\item
values of $\bm u$ at each vertex ${\bm x}\in \mathcal{V}$,
 $$
 \quad\bm{u}({\bm x});
 $$
\item
moments of two normal components of $\bm u$ on each edge $e\in \mathcal{E}$,
$$
\int_{e}(\bm{u}\cdot \bm{\nu}_{e, i})w \mathrm{d} \bm x, \quad \forall w\in \mathcal{P}_{p-2}(e), \ i=1, 2;
$$
\item
moments of the normal component of $\bm u$ on each face $f\in \mathcal F$,
$$
\int_{f}\left (\bm{u}\cdot \bm{\nu}_{f}\right ) {w}\mathrm{d} \bm x, \quad \forall {w}\in  \mathcal{P}_{p-3}(f);
$$
\item
integrals over each tetrahedron $K\in \mathcal{T}$,
$$
\int_{K}\bm{u}\cdot \bm{v}\mathrm{d} \bm x,  \quad\forall \bm{v}\in \NED_{p-1}^1(K).
$$
\end{itemize}
The dimension of $\mathcal{W}^{2}_{p}(\mathcal T)$ is
$$
\dim \mathcal{W}^{2}_{p}(\mathcal T)=3V+2(p-1)E+1/2(p-1)(p-2)F+\left(\frac{1}{2}p^{3}-p^{2}+\frac{1}{2}p-1\right)T.
$$

\begin{theorem}
The sequence \eqref{complex-HZ} is a complex which is exact on contractible domains. 
\end{theorem}
\begin{proof}
To show that \eqref{complex-HZ} is a complex, we only show that $\curl\mathcal{W}_{p+1}^{1}(\mathcal T)\subset \mathcal{W}_{p}^{2}(\mathcal T)$.
 This is a consequence of the identity
$$
\curl\bm{w}\cdot \bm{\nu}_f=\div_{f}(\bm{w}\times \bm{\nu}_f) \text{ for }\bm w\in \mathcal{W}_{p+1}^{1}(\mathcal T),
$$
where $\bm{\nu}_f$ is the unit normal vector of $f$. Therefore the degrees of freedom of $\mathcal{W}_{p+2}^{1}(\mathcal T)$ imply the continuity of the normal component of $\curl\bm{w}$, meaning that $\curl\mathcal{W}_{p+1}^{1}(\mathcal T)\subset \mathcal{W}_{p}^{2}(\mathcal T)$.

As concluded in \cite{christiansen2016nodal}$, \div: \mathcal{W}_{p+1}^{2}(\mathcal T)\mapsto \mathcal{W}_{p}^{3}(\mathcal T)$ is onto. Moreover,  since $\mathcal{W}_{p+2}^{1}(\mathcal T)$ is of $C^{1}$ continuity at vertices and of $C^{0}$ continuity on edges, the preimage of $\grad$ has to be $C^{2}$ at vertices and $C^{1}$ on edges.

Finally, a dimension count in the sense of \eqref{dimcount} implies the exactness.
\end{proof}

\begin{figure}[htbp]
\begin{center}
\includegraphics[width=0.68\linewidth]{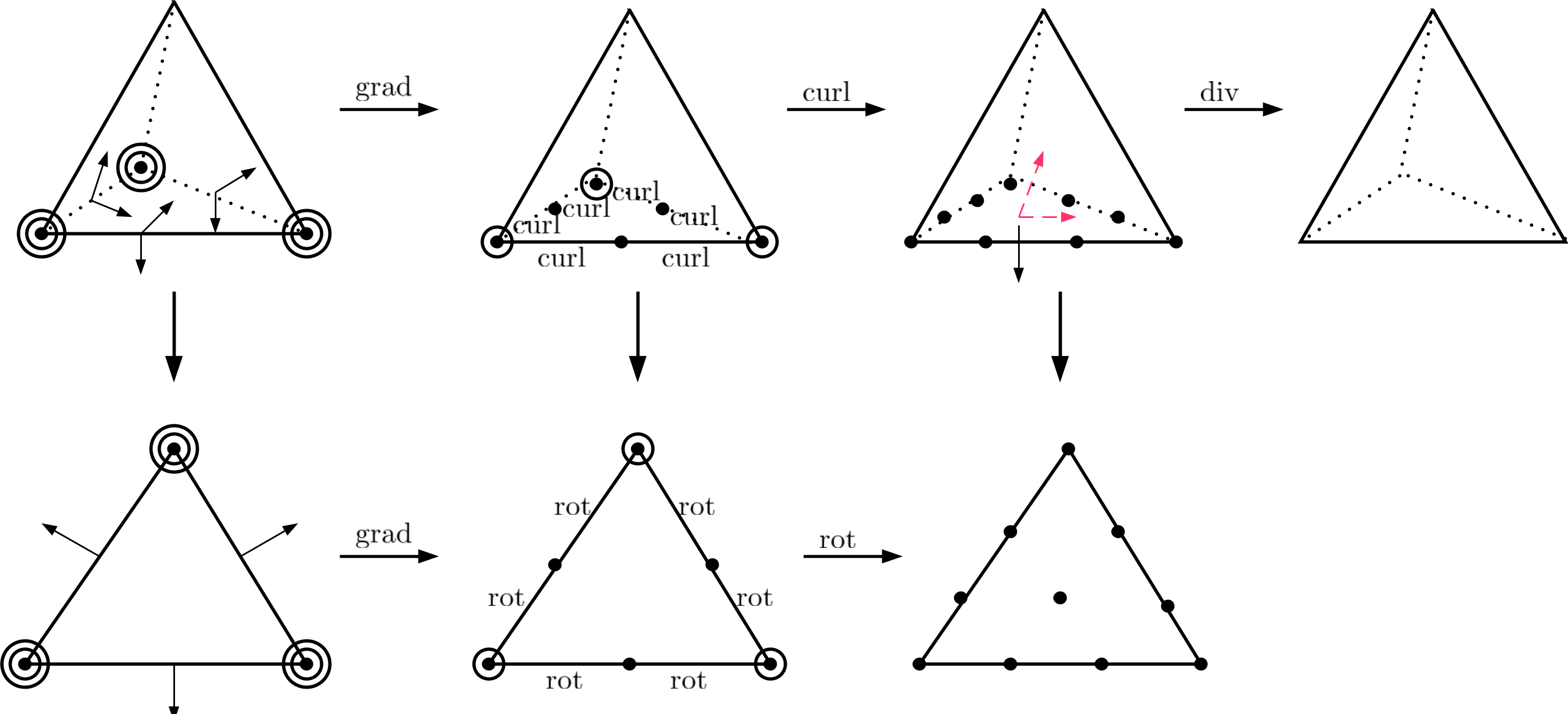}
\caption{The complex containing the 3D $H(\div)$ nodal element and its 2D restriction. For $H(\rot)$/$H(\curl)$/$H(\div)$ elements, a dot $\bullet$ represents the evaluation of the vector valued function, which counts as two/three degrees of freedom. In the 3D $H(\div)$ element, the two red dashed lines indicate the two tangential components associated with the Lagrange point on the face. They are regarded as interior degrees of freedom, which do not contribute to inter-element continuity.}
\label{fig:div}
\end{center}
\end{figure}

\begin{remark}
The rotation of the restriction of \eqref{complex-HZ} to each face (the second row in Figure \ref{fig:div}) was discussed in \cite{Arnold2006a} to interpret the Arnold-Winther elasticity element with the Bernstein-Gelfand-Gelfand (BGG) construction \cite{Arnold.D;Falk.R;Winther.R.2006a,arnold2021complexes}.
\end{remark}

\section{Partially orthogonal nodal bases}

To obtain practical bases for high order finite elements, one should take several issues into consideration, including the condition number, sparsity, symmetry and fast assembling. In particular, certain orthogonality is usually incorporated in the bases. Bases for high order Lagrange elements are relatively mature, see, e.g., \cite{karniadakis2013spectral}. In this section, we construct bases for partially discontinuous nodal $H(\curl)$ and $H(\div)$ elements in both 2D and 3D based on explicit formulas of multi-variate orthogonal polynomials on simplices. 
 We aim at orthogonality in the $L^{2}$ sense. The reason to use the $L^{2}$ orthogonality has been addressed in \cite{hu2017well}, see also \cite{ainsworth2021preconditioning,xin2012well}. 

 We choose proper Jacobi polynomials such that the bases associated to each simplex are mutually orthogonal on the reference element. 
As in the case of the $C^{0}$ elements, the $L^{2}$ condition numbers of the global mass and stiffness matrices still grow with the polynomial degree if no proper preconditioner is adopted .

To fix the notation, we use $J_{n}^{\alpha,\beta}(x)$ to denote the Jacobi polynomial on $[-1, 1]$ with weight $\omega^{\alpha,\beta}:=(1-x)^{\alpha}(1+x)^{\beta}$, i.e.
$$
\int_{-1}^{1}J_{n}^{\alpha, \beta}(x)J_{n'}^{\alpha,\beta}(x)(1-x)^{\alpha}(1+x)^{\beta}\, \mathrm{d}\bm x=\gamma_{n}^{\alpha,\beta}\delta_{nn'},
$$
where $\gamma_{n}^{\alpha,\beta}=\int_{-1}^{1} J_{n}^{\alpha,\beta}(x)^{2}(1-x)^{\alpha}(1+x)^{\beta}\, \mathrm{d}\bm x$. By a simple transform $\xi=(x+1)/2$, we obtain Jacobi polynomials on the reference interval $[0,1]$:
$$
\int_{0}^{1}J_{n}^{\alpha, \beta}(2\xi -1)J_{n'}^{\alpha,\beta}(2\xi -1)(1-\xi)^{\alpha}\xi^{\beta}\, \mathrm{d}\xi=c_{n}^{\alpha,\beta}\delta_{nn'},
$$
where 
$$
c_{n}^{\alpha,\beta}=2^{-\alpha-\beta-1}\gamma_{n}^{\alpha,\beta}.
$$

We cite the explicit form of Jacobi polynomials on the reference simplex of dimension $d$ (c.f. \cite{dunkl2014orthogonal}, Section 5.3):
\begin{theorem}
The following polynomials form a group of mutually orthonormal bases on the $d$-dimensional reference simplex with weight $\lambda_{0}^{\alpha_{0}}\dots \lambda_{d}^{\alpha_{d}}$:
\begin{align}\label{nd-orthogonal}
J_{\bm{n}}^{\bm{\alpha}}(\lambda_{0}, \dots, \lambda_{d-1}):=(c_{\bm{n}}^{\bm{\alpha}})^{-1}\prod_{j=0}^{d-1}\left ( 1-\sum_{i=0}^{j-1}\lambda_{i} \right )^{n_{j}}J_{n_{j}}^{a_{j}, b_{j}}\left(  \frac{2\lambda_{j}}{1-\sum_{i=0}^{j-1}\lambda_{i}} -1 \right ),
\end{align}
where $\bm{n}=(n_{0}, \dots, n_{d-1})$ and $\bm{\alpha}=(\alpha_{0}, \alpha_{1}, \dots, \alpha_{d})$ are multi-indexes of the polynomial degree and the weights, respectively,  $a_{j}=2\sum_{i=j+1}^{d-1}n_{i}+\sum_{i=j+1}^{d}\alpha_{i}+d-j-1$ and $b_{j}=\alpha_{j}$.  Here $c_{\bm{n}}^{\bm{\alpha}}$ is a scaling factor \cite{xin2011well}. When $J_{n_{j}}^{a_{j}, b_{j}}(\cdot)$ is scaled to be orthonormal, we have
$$
\left ( c_{\bm{n}}^{\bm{\alpha}}\right )^{-2}=\Pi_{i=0}^{d-1} 2^{a_{i-1}}.
$$
\end{theorem}
In 2D we obtain a group of mutually orthogonal bases as follows
\begin{align}\label{2d-orthogonal}
J_{\bm{n}}^{\bm{\alpha}}=(c_{\bm{n}}^{\bm{\alpha}})^{-1}J_{n_{0}}^{2n_{1}+\alpha_{1}+\alpha_{2}+1, \alpha_{0}}(2\lambda_{0}-1)J_{n_{1}}^{\alpha_{2}, \alpha_{1}}\left ( \frac{2\lambda_{1}}{1-\lambda_{0}}-1  \right )(1-\lambda_{0})^{n_{1}}.
\end{align}

From  \eqref{2d-orthogonal}, we can obtain two groups of univariate polynomials which are orthogonal on a triangle. Let $n_{1}=0$, we obtain the first group:
$$
cJ_{n_{0}}^{\alpha_{1}+\alpha_{2}+1, \alpha_{0}}(2\lambda_{0}-1),
$$
where $\lambda_{0}$ is considered as the variable. For the second, we let $n_{0}=0$ and obtain
$$
cJ_{n_{1}}^{\alpha_{2}, \alpha_{1}}\left ( \frac{2\lambda_{1}}{1-\lambda_{0}}-1  \right )(1-\lambda_{0})^{n_{1}},
$$
where $\lambda_{0}$ and $\lambda_{1}$ are considered as variables.
Note that $\lambda_{0}+\lambda_{1}+\lambda_{2}=1$ on a triangle, we get yet another formula with variables $\lambda_{1}$ and $\lambda_{2}$:
\begin{align}\label{scaled-Legendre}
cJ_{n_{1}}^{\alpha_{2}, \alpha_{1}}\left ( \frac{\lambda_{1}-\lambda_{2}}{\lambda_{1}+\lambda_{2}} \right )(\lambda_{1}+\lambda_{2})^{n_{1}}.
\end{align}
When $\alpha_{1}=\alpha_{2}=0$, \eqref{scaled-Legendre} is the scaled Legendre polynomials used in \cite{zaglmayr2006high}. Due to the symmetry of variables, we will use \eqref{scaled-Legendre} to build up edge based bases for triangular elements.

The 3D version of \eqref{nd-orthogonal} reads
\begin{align} \nonumber  \label{3d-orthogonal1}
J_{\bm{n}}^{\bm{\alpha}}=(c_{\bm{n}}^{\bm{\alpha}})^{-1}&J_{n_{0}}^{2(n_{1}+n_{2})+\alpha_{1}+\alpha_{2}+\alpha_{3}+2, \alpha_{0}}(2\lambda_{0}-1)
J_{n_{1}}^{2n_{2}+\alpha_{2}+\alpha_{3}+1, \alpha_{1}}\left ( \frac{2\lambda_{1}}{1-\lambda_{0}} -1 \right )
\\&
\cdot(1-\lambda_{0})^{n_{1}}J_{n_{2}}^{\alpha_{3}, \alpha_{2}}\left (\frac{2\lambda_{2}}{1-\lambda_{0}-\lambda_{1}} -1\right )(1-\lambda_{0}-\lambda_{1})^{n_{2}}.
\end{align}
We can similarly consider its restriction to an edge or a face, which will be used in the construction of edge and face bases for the tetrahedral elements. 

\subsection{Triangular $H(\rot)$ element}

\begin{enumerate}
\item
Vertex-based basis functions at $\bm{x}\in \mathcal{V}$:
$$
\bm{v}_{\bm{x}, i}=\lambda_{\bm x}\bm{e}_{i}, \quad i=1, 2,
$$
where $\lambda_{\bm x}$ is the barycentric coordinates at $\bm{x}$ and $\bm e_i$ is a 2D vector with the ith component 1 and the others 0.
\item
Edge-based tangential bases on $e\in \mathcal E$: 
\begin{align}\label{22index}
\bm{v}_{e, j}^{\tau}= \phi_{j}\bm{\tau}_{e},\quad j=0, 1, 2, \cdots,\ \text{with }
 \phi_{j}=J_{j}^{2, 2}\left ( \frac{\lambda_{1}-\lambda_{2}}{\lambda_{1}+\lambda_{2}} \right )(\lambda_{1}+\lambda_{2})^{j}\lambda_{1}\lambda_{2},\end{align}
where $\bm{\tau}_{e}$ is the tangential vector of the edge $e$ and $\lambda_{1},\ \lambda_{2}$ are the two barycentric coordinates of the two vertices of $e$.
\item
Edge-based normal bases on $e\in \mathcal E$: 
$$
\bm{v}_{e,i, j}^{\nu}= \phi_{j}|_{f_{i}}\bm{\nu}_{e}, \quad i=1,2,\quad j=0, 1, 2, \cdots
$$
where $f_{1}$ and $f_{2}$ are the two elements sharing the edge $e$ and $\bm{\nu}_{e}$ is the unit normal vector of $e$.
\item
Interior bases in $f\in \mathcal T$ (here, $\mathcal T$ is a triangular mesh of a 2D domain): 
$$
\bm{v}_{f,  i}^{n_{0}, n_{1}}=\phi_{n_{0}, n_{1}}\bm{e}_{i}, \quad i=1,2,
$$
where
$$
\phi_{n_{0}, n_{1}}=(c_{\bm{n}}^{\bm{\alpha}})^{-1}J_{n_{0}}^{2n_{1}+5, 2}(2\lambda_{0}-1)J_{n_{1}}^{2, 2}\left ( \frac{2\lambda_{1}}{1-\lambda_{0}}-1  \right )(1-\lambda_{0})^{n_{1}}\lambda_{0}\lambda_{1}\lambda_{2},
$$
and $\lambda_{0}, \lambda_{1}, \lambda_{2}$ are the three barycentric coordinates of the three vertices of $f$.
\end{enumerate}

\begin{remark}
The $\lambda_{1}\lambda_{2}$ factor appearing in \eqref{22index} is such that the bases decay to zero on the other two edges of the triangle. Then the $(2, 2)$ index in the Jacobi polynomials is such that the bases are $L^{2}$ mutually orthogonal. A similar construction will also be used below for tetrahedral elements. 
\end{remark}

\subsection{Tetrahedral $H(\curl)$ element on the Worsey-Farin split}

\begin{enumerate}
\item
Vertex-based basis functions at $\bm{x}\in \mathcal{V}$: 
$$
\bm{v}_{\bm{x}, i}=\lambda_{\bm x}\bm{e}_{i}, \quad i=1, 2, 3,
$$
where $\lambda_{\bm x}$ is the barycentric coordinates at $\bm{x}$ and $\bm e_i$ is a 3D vector with the ith component 1 and the others 0.
\item
Edge-based basis functions on $e\in \mathcal E$:
$$
\bm{v}_{e,j, i}= \phi_{j}\bm{e}_{i},\quad i=1,2, 3, \quad j=0, 1, 2, \cdots.
$$
Here
$$
 \phi_{j}=J_{j}^{2, 2}\left ( \frac{\lambda_{1}-\lambda_{2}}{\lambda_{1}+\lambda_{2}} \right )(\lambda_{1}+\lambda_{2})^{j}\lambda_{1}\lambda_{2},
$$
where  $\lambda_{1}$ and $\lambda_{2}$ are the barycentric coordinates of the two vertices of $e$.
\item
Face-based normal bases on $f\in \mathcal F$, which is shared by two macroelements $K_{1}$ and $K_{2}$:
\begin{itemize}
\item at the face subdivision point,
$$
\lambda|_{K_{i}} \bm{\nu}_{f}, \quad i=1, 2,
$$
where $\lambda|_{K_{i}}$ is the restriction of the hat function on one macroelement (a piecewise linear function on each macroelement);
\item on a face-interior edge $e_k$ of $f$,
$$
\bm{v}_{j,i,k}= \phi_{j}^{(k)}|_{K_{i}}\bm{\nu}_{f}, \quad i=1,2, \ \ k=1,2,3, \ \ j=0, 1, 2, \cdots,
$$
where
$$
 \phi_{j}^{(k)}=J_{j}^{2, 2}\left ( \frac{\lambda_{1}^{(k)}-\lambda_{2}^{(k)}}{\lambda_{1}^{(k)}+\lambda_{2}^{(k)}} \right )(\lambda_{1}^{(k)}+\lambda_{2}^{(k)})^{j}\lambda_{1}^{(k)}\lambda_{2}^{(k)}
$$
with  the barycentric coordinates $\lambda_{1}^{(k)}$ and $\lambda_{2}^{(k)}$  of the two vertices of $e_k$;
\item in the interior of a micro face $f_j$ of $f$,
$$
\bm{v}_{f, i,j}^{n_{1}, n_{2}}=\left. \phi^{(j)}_{n_{1}, n_{2}}\right |_{K_{i}}\bm{\nu}_{f}, \quad i=1,2, \ \ j=1,2,3,
$$
\begin{align*}
\phi^{(j)}_{n_{1}, n_{2}}=&J_{n_{1}}^{2n_{2}+5, 2}\left ( \frac{\lambda_{1}^{(j)}-\lambda_{2}^{(j)}-\lambda_{3}^{(j)}}{\lambda_{1}^{(j)}+\lambda_{2}^{(j)}+\lambda_{3}^{(j)}}  \right )
\cdot(\lambda_{1}^{(j)}+\lambda_{2}^{(j)}+\lambda_{3}^{(j)})^{n_{1}}\cdot\\
&J_{n_{2}}^{2, 2}\left ( \frac{\lambda_{1}^{(j)}-\lambda_{2}^{(j)}}{\lambda_{1}^{(j)}+\lambda_{2}^{(j)}} \right )(\lambda_{1}^{(j)}+\lambda_{2}^{(j)})^{n_{2}}\lambda_{1}^{(j)}\lambda_{2}^{(j)}\lambda_{3}^{(j)},
\end{align*}
with barycentric coordinates $\lambda_{1}^{(j)}$, $\lambda_{2}^{(j)}$, $\lambda_{3}^{(j)}$ of each micro face $f_j$ of $f$.
\end{itemize}
\item
Face-based tangential bases on $f\in \mathcal F$:
\begin{itemize}
\item at the face subdivision point,
$$
\lambda \bm{\tau}_{f, i}, \quad i= 1, 2,
$$
where $\lambda$ is the hat function at the face subdivision point and $\bm{\tau}_{f, i}, i=1, 2,$ are the two linearly independent unit tangent vectors of the face $f$; 
\item on a face-interior edge $e_k$,
$$
\bm{v}_{i,j,k}= \phi_{j}^{(k)} \bm{\tau}_{f, i}, \quad i= 1, 2, \ \ k=1,2,3\ \ j=0, 1, 2, \cdots
$$
\item in the interior of each micro face $f_j$ of $f$,
$$
\bm{v}_{f, i,j}^{n_{1}, n_{2}}= \phi_{n_{1}, n_{2}}^{(j)}\bm{\tau}_{f, i}, \quad i=1,2, \ \ j=1,2,3.
$$
\end{itemize}
\item
Interior bases: 
\begin{itemize}
\item at the subdivision point,
$$
\lambda_{Q}\bm{e}_{i}, \quad i=1, 2, 3,
$$
where $\lambda_{Q}$ is the hat function at the interior subdivision point;
\item on an interior edge $e_k$,
$$
\bm{v}_{j, i,k}= \phi_{j}^{(k)}\bm{e}_{i}, \ \ i=1, 2, 3, \ \ k=1,2,\cdots, 8, \ \  j=0, 1, 2, \cdots
$$
\item in the interior of an interior micro face $f_j$,
$$
\bm{v}_{i,j}^{n_{1}, n_{2}}= \phi_{n_{1}, n_{2}}^{(j)}\bm{e}_{i}, \quad i=1,2, 3,\ \ j=1,2,\cdots, 12,
$$
where
$$
\phi^{(j)}_{n_{1}, n_{2}}=J_{n_{1}}^{2n_{2}+5, 2}\left ( \frac{\lambda_{1}^{(j)}-\lambda_{2}^{(j)}-\lambda_{3}^{(j)}}{\lambda_{1}^{(j)}+\lambda_{2}^{(j)}+\lambda_{3}^{(j)}}  \right )
\cdot(\lambda_{1}^{(j)}+\lambda_{2}^{(j)}+\lambda_{3}^{(j)})^{n_{1}}J_{n_{2}}^{2, 2}\left ( \frac{\lambda_{1}^{(j)}-\lambda_{2}^{(j)}}{\lambda_{1}^{(j)}+\lambda_{2}^{(j)}} \right )(\lambda_{1}^{(j)}+\lambda_{2}^{(j)})^{n_{2}}\lambda_{1}^{(j)}\lambda_{2}^{(j)}\lambda_{3}^{(j)},
$$
with $\lambda_{1}^{(j)}$, $\lambda_{2}^{(j)}$ and $\lambda_{3}^{(j)}$ the barycentric coordinates associate with the interior face $f_j$;
\item in the interior of the micro tetrahedra $K_j$, 
$$
\bm{v}_{i,j}^{n_{0}, n_{1}, n_{2}}=\phi_{n_{0}, n_{1}, n_{2}}^{(j)}\bm{e}_{i}, \quad i=1,2, 3,\ \ j=1,2,\cdots,12,
$$
where
\begin{align} \nonumber 
&\phi^{(j)}_{n_{0}, n_{1}, n_{2}}=J_{n_{0}}^{2(n_{1}+n_{2})+8, 2}(2\lambda_{0}^{(j)}-1)
J_{n_{1}}^{2n_{2}+5, 2}\left ( \frac{2\lambda_{1}^{(j)}}{1-\lambda_{0}^{(j)}} -1 \right )(1-\lambda_{0}^{(j)})^{n_{1}}
\\&\label{3Dbases}
\cdot J_{n_{2}}^{2, 2}\left (\frac{2\lambda_{2}^{(j)}}{1-\lambda_{0}^{(j)}-\lambda_{1}^{(j)}} -1\right )(1-\lambda_{0}^{(j)}-\lambda_{1}^{(j)})^{n_{2}}\lambda_{0}^{(j)}\lambda_{1}^{(j)}\lambda_{2}^{(j)}\lambda_{3}^{(j)},\end{align}
and $\lambda_i^{(j)},i=0,1,2,3$  are the barycentric coordinates associate with the micro tetrahedron $K_j$. 
\end{itemize}
\end{enumerate}

\subsection{Tetrahedral $H(\div)$ elements}

\begin{enumerate}
\item
Vertex-based basis functions at $\bm{x}\in \mathcal{V}$: 
$$
\bm{v}_{\bm{x}, i}=\lambda_{\bm x}\bm{e}_{i}, \quad i=1, 2, 3.
$$
\item
Edge-based tangential bases on $e\in \mathcal E$: 
$$
\bm{v}_{e, i, j}^{\tau}= \phi_{j}|_{K_{i}}\bm{\tau}_{e}, \ j=0,1,\cdots, 
$$
where $\bm{\tau}_{e}$ is the tangent vector of $e$, $K_{i}$ takes $e$ as an edge.
\item
Edge-based normal bases on $e\in \mathcal E$:
$$
\bm{v}_{e,  i, j}^{\nu}= \phi_{j}\bm{\nu}_{e, i},\quad i=1, 2, \ \ j=0,1,\cdots,
$$
where $\bm{\nu}_{e, i}, i=1, 2,$ are the two normal vectors of $e$.
\item
Face-based tangential bases on $f\in \mathcal F$: $$
\bm{v}_{f, \tau, i}^{n_{1}, n_{2}}= \phi_{n_{1}, n_{2}}|_{K_{i}}\bm{\tau}_{{f}, i},\quad i=1,2,
$$
where $\bm{\tau}_{f, i}, i=1, 2,$ are the two tangent vectors of $f$,  $K_{1}$ and $K_{2}$ are the two elements  sharing $f$. Moreover, 
$$
\phi_{n_{1}, n_{2}}=J_{n_{1}}^{2n_{2}+5, 2}\left ( \frac{\lambda_{1}-\lambda_{2}-\lambda_{3}}{\lambda_{1}+\lambda_{2}+\lambda_{3}}  \right )
\cdot(\lambda_{1}+\lambda_{2}+\lambda_{3})^{n_{1}}J_{n_{2}}^{2, 2}\left ( \frac{\lambda_{1}-\lambda_{2}}{\lambda_{1}+\lambda_{2}} \right )(\lambda_{1}+\lambda_{2})^{n_{2}}\lambda_{1}\lambda_{2}\lambda_{3},
$$
where $\lambda_{1}$, $\lambda_{2}$, and $\lambda_{3}$ are the barycentric coordinates associated with $f$;
\item
Face-based normal bases  on $f\in \mathcal F$:
$$
\bm{v}_{f,\nu}^{n_{1}, n_{2}}= \phi_{n_{1}, n_{2}}\bm{\nu}_{f},
$$
where $\bm{\nu}_{f}$ is the normal vector of $f$.
\item
Interior basis functions  in $K\in \mathcal T$:
$$
\bm{v}_{K,i}^{n_{0}, n_{1}, n_{2}}=\phi_{n_{0}, n_{1}, n_{2}}\bm{e}_{i}, \quad i=1,2,3,
$$
where $\phi_{n_{0}, n_{1}, n_{2}}$ 
is given by \eqref{3Dbases} without the superscript $(j)$.
\end{enumerate}

\subsection{Condition number}

We test the condition number of mass and quasi-stiffness matrices. The condition number of the mass matrix $M$ is calculated by 
\[\kappa(A)=\frac{\lambda_{\max}}{\lambda_{\min}},\]
where $\lambda_{\max}$, $\lambda_{\min}$ are the maximum and minimum eigenvalues of the matrix $M$, respectively. For the quasi-stiffness matrix $S$, only nonzero eigenvalues will be considered.
We also consider the normalized mass and quasi-stiffness matrices $\widetilde M$ and $\widetilde S$, whose nonzero diagonal entries are scaled to 1. The condition numbers of the four matrices are listed in Table \ref{tab4}. The condition numbers of the normalized mass and quasi-stiffness matrices are $O(10^4)$ and $O(10)$ when the $p$ reaches 9 and  the condition number of $\tilde{S}$ is almost constant.

\begin{table}[!ht]
  \caption{condition numbers of $M$, $\widetilde M$, $S$, and $\widetilde S$.} \label{tab4}
   \centering
    \begin{adjustwidth}{3cm}{0cm}
  \resizebox{!}{1.8cm}
  {
\begin{tabular}{cccccccccccc}
   \hline
{$p$}& $M$&  $\widetilde M$& $S$ &  $\widetilde S$
    \\ \hline
 2&3.0345e+02&1.0121e+02& 7.8806e+01&2.9527e+01\\\hline
 3&9.1025e+03&2.2944e+02&3.2677e+02&5.8554e+00\\\hline
4&7.7549e+05&1.9106e+03& 9.4503e+03&1.9563e+01\\\hline
5&6.9908e+06&2.7049e+03&1.3638e+05&1.3648e+01\\\hline
6& 1.1922e+08& 5.3600e+03&  1.6570e+06& 1.4868e+01\\\hline
7& 1.9673e+09& 7.4625e+03&  2.0611e+07& 1.6779e+01\\\hline
9& 3.2159e+10& 1.2387e+04&  2.5840e+08&  1.9111e+01\\\hline
    \end{tabular}
    }
    \end{adjustwidth}
    \end{table}

\section{Conclusions}
This paper constructed partially discontinuous nodal finite elements for $H(\curl)$ and $H(\div)$ by allowing discontinuity in the tangential or normal directions of vector-valued Lagrange elements, respectively. These elements can be implemented as a combination of standard continuous and discontinuous elements. The bases appear to be well-conditioned in the numerical tests, which shows the computational potential of these finite elements. Other bases, local complete sequences \cite{schoberl2005high} and $hp$- preconditioning can be investigated in future research.

\bibliographystyle{siam}      
\bibliography{references}{}   

\end{document}